\documentclass[12pt,reqno]{amsart}

\usepackage{hyperref} 

\usepackage[all]{xy}

\xymatrixcolsep{2cm}

\numberwithin{equation}{section}

\allowdisplaybreaks[1]

\usepackage{etoolbox}

\makeatletter

\patchcmd{\thesubsection}{\arabic}{\Alph}{}{}

\patchcmd{\@seccntformat}{\@secnumfont}{%
\@secnumfont\expandafter\protect\csname format#1\endcsname}{}{}

\patchcmd{\@startsection}{\@afterindenttrue}{\@afterindentfalse}{}{}

\patchcmd{\subsection}{-.5em}{.3\linespacing}{}{}

\makeatother

\theoremstyle{plain}

\newtheorem{theorem}{Theorem}[section]
\newtheorem{question}[theorem]{Question}

\newtheorem{proposition}[theorem]{Proposition}

\newtheorem{lemma}[theorem]{Lemma}
\newtheorem{example}[theorem]{Example}

\theoremstyle{remark}

\newtheorem{remark}[theorem]{Remark}

\newcommand{\Ker}[1]{\ensuremath{\mathrm{Ker} (#1)}}

\newcommand{\cat}[1]{\ensuremath{\mathcal{#1}}}

\newcommand{\at}[2][]{\ensuremath{\mathrm{at}_{#1} (#2)}}

\newcommand{\END}[2][]{\ensuremath{\mathcal{E}\mathit{nd}_{#1} (#2)}}

\newcommand{\id}[1]{\ensuremath{\mathbf{1}_{#1}}}

\newcommand{\p}{\ensuremath{\mathbb{P}}}

\newcommand{\ad}[1]{\ensuremath{\mathrm{ad}(#1)}}

\newcommand{\C}{\ensuremath{\mathbb{C}}}

\newcommand{\struct}[1]{\ensuremath{\mathcal{O}_{#1}}}

\newcommand{\Lie}[1]{\ensuremath{\mathrm{Lie}(#1)}}

\newcommand{\coh}[3]{\ensuremath{\mathrm{H}^{#1}(#2,\,#3)}}

\begin{document}

\title[relative Connections on Principal bundles]{relative connections on Principal bundles and relative equivariant structures}

\author[M. Poddar]{Mainak Poddar}

\address{Department of Mathematics, Indian Institute of Science Education and Research (IISER), Pune, India}
\email{mainak@iiserpune.ac.in}

\author[A. Singh]{Anoop Singh}

\address{School of Mathematics, Tata Institute of Fundamental Research,
Homi Bhabha Road, Mumbai 400005, India}

\address{Department of Mathematical Sciences, Indian Institute of Technology (BHU), Varanasi  221005, India}

\email{anoopsingh5192@gmail.com}



\begin{abstract}
We investigate relative holomorphic connections on a principal bundle over a family of compact complex manifolds. A sufficient condition is given for the 
existence of a relative holomorphic connection on a holomorphic principal bundle over a complex analytic family. We also introduce the notion of relative 
equivariant bundles and establish its relation with 
relative holomorphic connections on principal bundles.
\end{abstract}
\maketitle

\section{Introduction}
\label{Intro}
The notion of holomorphic connections on a holomorphic vector bundle was introduced by Atiyah \cite{A}, which was further generalised in many contexts in mathematics.
A well-known theorem due to Atiyah \cite{A} and Weil \cite{W} says that a holomorphic vector bundle $E$ over a compact Riemann
surface $Y$ admits a holomorphic connection if and only if the degree of every holomorphic direct summand
of $E$ is zero.  In \cite{AB}, this result was extended to holomorphic principal $G$-bundles on $Y$, where $G$ is a connected reductive complex algebraic group. Moreover, in \cite{BKN} and \cite{BP}, authors have studied the relationship between the existence of equivariant structures and holomorphic $G$-connections on a principal bundle over a 
complex manifold.

Throughout this article $\pi : X \to S$ will denote a 
surjective holomorphic proper submersion between two complex manifolds $X$ and $S$
with connected fibers.

Motivated by the above results, we have a basic question in 
the relative set up described as follows.

\begin{question}\label{qn} Let $H$ be a connected complex Lie group.
Let $E^H \,\stackrel{\varpi}{\longrightarrow}\, X \xrightarrow{\pi} S$ be a
holomorphic principal $H$-bundle over $X/S$.
Is there a good criterion for the existence of a relative holomorphic connection on 
$E^H$?
\end{question}

We tackle above question in the following manner (see Section  \ref{rel_Atiyah}).

Let $H$ and $E^H$ be as above.
Then, we construct a short exact sequence 
\begin{equation}
 \label{eq:0.9}
 0 \,\longrightarrow\, \text{ad}(E^H)\, \stackrel{\iota}{\longrightarrow}\,
\text{At}_S(E^H) \, \stackrel{ \widetilde{d \varpi} }{\longrightarrow}\,  T_{X/S} \,\longrightarrow\, 0,
 \end{equation}
 of vector bundles over $X/S$, where $\ad{E^H}$ is the adjoint vector bundle for $E^H$ and $\text{At}_S(E^H)$
 is the relative  Atiyah bundle
 for $E^H$  (see \eqref{eq:7.2}). 
 
 A relative holomorphic connection on $E^H$ is by definition  a holomorphic splitting of \eqref{eq:0.9},
 see Lemma \ref{lem:1} for equivalent conditions.

We give a sufficient condition for the existence of a relative holomorphic connection on $E^H$ (see Theorem \ref{thm:1}), more precisely, we prove the following.

Suppose that for every
$s \in S$, there is a holomorphic connection on the principal $H$-bundle
$\varpi|_{E^H_s}\, :\, E^H_s \,\longrightarrow\, X_s$, and
$$\coh{1}{S}{\pi_*(\Omega^1_{X/S} \otimes \ad{E^H})} \,=\, 0\, .$$
Then, $E^H$ admits a relative holomorphic connection.

We also note in Proposition \ref{prop:1} that the existence of a holomorphic connection on each  bundle $E_s^H$, $s \in S$, is a necessary condition for the existence of a relative holomorphic connection on $E^H$.

Let $G$ be a complex Lie group and let
$\pi : X \longrightarrow S$
be  of relative dimension $l = m-n$, that is, 
$X$ is a complex analytic family of connected complex manifolds of dimension $l$ parametrised by a complex manifold $S$ of
dimension $n$.
For every point $s \in S$, we denote $\pi^{-1}(s)$ by 
$X_s$.
Consider actions of $G$ on $X$
\begin{equation*}
 \label{eq:0.10}
 \tau : G \times X \to X \,,
 \end{equation*}
 and  on $S$
 \begin{equation*}
\label{eq:0.24.1}
\nu : G \times S \to S \,,
\end{equation*}
such that $\pi : X \to S$ is $G$-equivariant.
 
 A similar question as in Question \ref{qn}
can be asked for the existence of  relative holomorphic $G$-connections on $E^H$. For that, we proceed as follows (see Section \ref{rel_Atiyah}).

Given the action of $G$ on $X$ and $S$ such that $\pi : X \to S$ is $G$-invariant (i.e., the action $\nu$ is trivial), we also
construct a short exact sequence of holomorphic vector bundles over $X/S$
\begin{equation}
 \label{eq:0.17}
 0 \,\longrightarrow\, \text{ad}(E^H)\, \stackrel{\iota_0}{\longrightarrow}\,
\text{At}_S^\tau(E^H) \, \stackrel{q }{\longrightarrow}\,  X \times \mathfrak{g} \,\longrightarrow\, 0,
 \end{equation}
 where $\mathfrak{g}$ is the Lie algebra of $G$. 
 The vector bundle $\text{At}_S^\tau(E^H)$ mentioned in 
 \eqref{eq:0.17} is a subbundle of the vector bundle 
 $\text{At}_S(E^H) \oplus (X \times \mathfrak{g})$. By definition,
 a relative holomorphic $G$-connection on $E^H$ is the
 holomorphic  splitting of the  short exact sequence 
 \eqref{eq:0.17}. We prove a sufficient condition for the existence of the relative holomorphic $G$-connection
 on $E^H$ (see Theorem \ref{thm:2}). Again, the existence of a holomorphic $G$-connection on each $E^H_s$ is a necessary condition
 (see Proposition \ref{prop:2}), and a part of the sufficient condition, for the existence of a relative holomorphic $G$-connection on $E^H$. 
 
 To illustrate Theorem \ref{thm:1} and Theorem \ref{thm:2}, we give  examples of the existence of relative 
 holomorphic connections and $G$-connections on $E^H$ where $S$ is a Stein manifold (see Example \ref{exmp:Stein}). We also give an example of a relative $G$-connection on $E^H$ where $S$ is a projective space and $H$ is an abelian complex Lie group (see Example \ref{exmp:rel}).

Note that the notion of a $G$-connection on $E^H$ depends on the $G$-action on the base $X$. This is useful as a $G$-connection can then serve as a tool to determine if a $G$-action on $X$ can be lifted to a $G$-action on $E^H$.
 In fact, the authors of \cite{BKN} and \cite{BP} have studied $G$-equivariant structure  on a principal $H$-bundle over a connected complex manifold from a connection theoretic perspective. Inspired by their work, our aim  in this article is to study the relative aspect of such $G$-equivariant structure on a family of principal $H$-bundles using relative $G$-connections.  

In section \ref{lift}, we consider the group  $\text{Aut}_S(E^H) $ of relative automorphisms  of $E^H$
over $X/S$. 
Let 
\begin{equation*}
\label{eq:0.28}
G_S \subset G
\end{equation*}
be the subset consisting of all $g \in G$ such that for 
every $s \in S$ the pulled back principal $H$-bundle 
$\tau_g^* E^H_{\nu_g(s)}$ is isomorphic to $E^H_s$.

Let $\cat{G}_S$ denote the space of all pairs of the form 
$(\theta, g)$ where $g \in G_S$, and 
$\theta : E^H \longrightarrow E^H$ is a holomorphic automorphism such that for every $s \in S$, 
$\theta_s : E^H_s \to E^H_{\nu_g(s)}$ is an isomorphism
over $\tau_g : X_s \to X_{\nu_g(s)}$. 
Under the assumption that $\pi: X \to S$ is $G$-invariant, we show that the Lie algebra of $\cat{G}_S$ is canonically identified  with the Lie algebra $\coh{0}{X}{\text{At}_S^\tau(E^H)}$ (see Proposition \ref{prop:4}). We also show that the holomorphic principal $H$-bundle $E^H$ admits a tautological
relative holomorphic $\cat{G}_S$-connection. The relative curvature of this  relative holomorphic $\cat{G}_S$-connection on $E^H$ vanishes identically (see Proposition \ref{prop:6}).

In section \ref{equi}, we define the relative equivariant structure on the principal $H$-bundle 
$E^H$ with respect to the given group $G$. We denote the 
action of $G$ on $E^H$ by $\sigma^E$ (see \eqref{eq:34}), and relative equivariant structure by the pair 
$(E^H, \sigma^E)$. When $\pi : X \to S$ is $G$-invariant,
we show that for a given relative equivariant structure 
$(E^H, \sigma^E)$ over $X/S$, $E^H$ admits a tautological relative holomorphic $G$-connection and the relative curvature of this relative holomorphic $G$-connection vanishes identically (see Proposition \ref{prop:7}). Under the assumption that $\pi : X \to S$ is $G$-invariant, and 
that $G$ is a semisimple and simply connected
affine algebraic group defined over $\C$, we also show that (see Theorem \ref{thm:3}), if 
$E^H \,\stackrel{\varpi}{\longrightarrow}\, X \xrightarrow{\pi} S$
admits a relative holomorphic $G$-connection $h$, then $E^H$
admits a relative equivariant structure 
$\sigma^E : G \times E^H \longrightarrow E^H.$

\subsection*{Acknowledgments} Both the authors would like to express their gratitude to Indranil Biswas for his collaboration with them on closely related topics and for generously sharing his knowledge. They also thank the anonymous referee for interesting comments and questions that helped to improve the manuscript.  The research of the first named author was supported in part by the SERB MATRICS Grant MTR/2019/001613. The second named author thanks IISER Pune for their hospitality while the work was carried out.

\section{Relative Atiyah sequence and group action }
\label{rel_Atiyah}
\subsection{Relative Atiyah exact sequence of a principal $H$-bundle}
Let $X$ and $S$ be two complex manifolds of dimensions $m$ and $n$ respectively.
Let 
\begin{equation}
\label{eq:0}
\pi : X \longrightarrow S
\end{equation}
be a holomorphic  surjective submersion of relative dimension $l = m-n$
such that the fibres are connected, that is, 
$X$ is a complex analytic family of connected complex manifolds of dimension $l$ parametrised by $S$.
For every point $s \in S$, we denote $\pi^{-1}(s)$ by 
$X_s$.

Let $H$ be a connected complex Lie group. We denote by 
$\mathfrak{h}$ its Lie algebra.

By a family of holomorphic principal $H$-bundles parametrised by $S$, we mean a holomorphic principal $H$-bundle
\begin{equation}
\label{eq:1}
\varpi : E^{H} \longrightarrow X
\end{equation}
over $X$ such that for every $s \in S$ the restriction
\begin{equation}
\label{eq:2}
\varpi \vert_{X_s}: E^H_s := E^H \vert_{X_s} \longrightarrow X_s
\end{equation}
is a holomorphic principal $H$-bundle over $X_s$.
Note that $H$ acts on both $X$ and $S$ trivially.

Let $d\pi\,:\, TX \,\longrightarrow\, \pi^*TS$ be the differential of $\pi$ in \eqref{eq:0}, where $TX$ and $TS$ be the holomorphic tangent bundles of $X$ and $S$
respectively.
The subbundle $$T_{X/S}\, :=\, \Ker{d\pi}\, \subset\, TX$$ is called the relative tangent
bundle for $\pi$. Thus we have a short exact sequence
of vector bundles
\begin{equation}\label{eq:ses1}
0 \,\longrightarrow\, T_{X/S}\, \stackrel{\imath}{\longrightarrow}\,
TX \, \stackrel{d\pi}{\longrightarrow}\, \pi^* TS \,\longrightarrow\, 0
\end{equation}
over $X$.

Consider the composition
\begin{equation}
\label{eq:3}
\pi \circ \varpi : E^{H} \longrightarrow S.
\end{equation}
Let 
\begin{equation}
\label{eq:4}
d (\pi \circ \varpi) : T E^{H} \longrightarrow (\pi \circ \varpi)^* TS
\end{equation}
be the differential of $\pi \circ \varpi$ in \eqref{eq:3}, where $T E^H$ is the holomorphic tangent bundle of 
$E^H$.
Its kernel
$$T_{E^H / S} :=  \Ker{d (\pi \circ \varpi)}$$
is known as relative tangent bundle for $\pi \circ \varpi$.
Moreover,
the restriction of the differential 
$$d \varpi : TE^H \longrightarrow \varpi^* TX$$
of $\varpi$ in \eqref{eq:1} to $T_{E^H/S}$ gives a morphism of bundles
\begin{equation}
\label{eq:5}
d \varpi' := (d \varpi )\vert_{T_{E^H/S}} : T_{E^H/S} \longrightarrow \varpi^* T_{X/S}
\end{equation} over $E^H$.
We denote its kernal by 
$$T_{E^H/X/S} := \Ker{(d \varpi) \vert_{T_{E^H/S}}}.$$
 So, we get a short exact sequence of vector bundles 
 \begin{equation}
 \label{eq:6}
 0 \,\longrightarrow\, T_{E^H/X/S}\, \stackrel{\imath}{\longrightarrow}\,
T_{E^H/S} \, \stackrel{d \varpi' }{\longrightarrow}\, \varpi^* T_{X/S} \,\longrightarrow\, 0
 \end{equation}
 over $E^H$.
 
 Note that we also have relative bundle for $\varpi$
 denoted as $T_{E^H/X} := \Ker{d \varpi}$. Moreover,
 \begin{equation}
 \label{eq:6.1}
 T_{E^H/X} \subset T_{E^H/S},
 \end{equation}
 and 
 \begin{equation}
 \label{eq:6.2}
 T_{E^H/X/S} \cong T_{E^H/X}.
 \end{equation}
 
 Let 
 \begin{equation}
 \label{eq:7}
 \sigma : E^H \times H \longrightarrow E^H
 \end{equation}
 be the action of $H$ on $E^H$. Note that the action of 
 $H$ on each fibre of $\varpi$ is free and transitive.
 The differential of $\sigma$ in \eqref{eq:7} induces a 
 homomorphism from the trivial vector bundle on $E^H$
 with fibre $\mathfrak{h}$
 $$E^H \times \mathfrak{h} \longrightarrow TE^H,$$
 and we have an isomorphism 
\begin{equation}
\label{eq:7.1}
T_{E^H/X/S} \cong E^H \times \mathfrak{h}
\end{equation} 
 of vector bundles over $E^H$.
 
 The action $\sigma$ in \eqref{eq:7} induces an action of $H$ on the total space of relative tangent bundle 
 $T_{E^H/S}$.
 The quotient 
 \begin{equation}
 \label{eq:7.2}
 \text{At}_S(E^H) := (T_{E^H/S})/H
 \end{equation}
is a holomorphic vector bundle over $X/S$, which is known as relative Atiyah bundle (see \cite{A}, \cite{BS}).

There is an adjoint action of $H$ on its Lie algebra $\mathfrak{h}$, which will induce an action of $H$ on the vector bundle 
$E^H \times \mathfrak{h}$.
Consider the quotient 
\begin{equation}
\label{eq:8}
\text{ad}(E^H) := E^H \times^H \mathfrak{h} = E^H \times \mathfrak{h}/H
\end{equation} 
which is known as adjoint vector bundle associated to $E^H$. 
From the identification \eqref{eq:7.1}, we have 
$$\text{ad}(E^H) = T_{E^H/X/S}/ H.$$

Thus, after taking the quotient by $H$, the short exact 
sequence in \eqref{eq:6} gives a short exact sequence  of holomorphic vector bundles 
\begin{equation}
 \label{eq:9}
 0 \,\longrightarrow\, \text{ad}(E^H)\, \stackrel{\iota}{\longrightarrow}\,
\text{At}_S(E^H) \, \stackrel{ \widetilde{d \varpi} }{\longrightarrow}\,  T_{X/S} \,\longrightarrow\, 0
 \end{equation}
 over $X$, which is known as {\it relative Atiyah exact 
 sequence} \cite{BS}, where $\widetilde{d \varpi}$ is 
 given by 
 $d \varpi'$ in \eqref{eq:5}.
 
 A {\it relative holomorphic connection} on $E^H$ is a 
 holomorphic splitting of the relative Atiyah exact 
 sequence in \eqref{eq:9}.
 
 Tensoring the short exact sequence in \eqref{eq:9} by
 the relative cotangent bundle $\Omega^1_{X/S}$, we get
 the following short exact sequence 
 \begin{equation}
 \label{eq:9.1}
 0 \,\longrightarrow\, \Omega^1_{X/S} \otimes \text{ad}(E^H)\, \stackrel{\iota}{\longrightarrow}\,
\Omega^1_{X/S} \otimes \text{At}_S(E^H) \, \stackrel{ \widetilde{d \varpi} }{\longrightarrow}\,  \END[\struct{X}]{T_{X/S}} \,\longrightarrow\, 0
 \end{equation}
 
 The above short exact sequence \eqref{eq:9.1} of \struct{X}-modules gives a long exact sequence of $\C$-vector spaces
\begin{equation}
\label{eq:9.2}
\cdots \to  \coh{0}{X}{\Omega^1_{X/S} \otimes \text{At}_S(E^H)} \to \coh{0}{X}{\END[\struct{X}]{T_{X/S}}} \xrightarrow{\delta}
\coh{1}{X}{\Omega^1_{X/S} \otimes \text{ad}(E^H)} \to \cdots,
\end{equation} 
where $\delta$ is the connecting homomorphism.
Now, the extension class of the relative Atiyah exact 
sequence is defined  by
\begin{equation}
\label{eq:9.3}
\at[S]{E^H} := \delta(\id{T_{X/S}}) \in \coh{1}{X}{\Omega^1_{X/S} \otimes \text{ad}(E^H)},
\end{equation}
which is also known as relative Atiyah class of the bundle $E^H$.

\begin{lemma}
\label{lem:1}
Let $E^H \,\stackrel{\varpi}{\longrightarrow}\, X \xrightarrow{\pi} S$
be a holomorphic principal $H$-bundle. Then, the followings are equivalent.
\begin{enumerate}
\item $E^H$ admits a relative holomorphic connection.
\item The relative Atiyah exact sequence for $E^H$ in 
\eqref{eq:9} splits.
\item The relative Atiyah class $\at[S]{E^H}$ vanishes.
\end{enumerate}
\end{lemma}

 \begin{proposition}[family of holomorphic connections]
 \label{prop:1}
 Suppose that $E^H$ admits a relative holomorphic connection. Then, we have a family of holomorphic connections on $\{E^H_s\}_{s\in S}$.
 \end{proposition}
 \begin{proof}
 The proof easily follows from the following commutative 
 diagram
\begin{equation}
\label{eq:cd1}
\xymatrix{
0 \ar[r] & \ad{E^H} \ar[d]^{r_s} \ar[r]^{\iota} & \text{At}_S(E^H) 
\ar[d]^{r_s} \ar[r]^{ \widetilde{d \varpi}} & T_{X/S} \ar[d]^{r_s} \ar[r] & 0 \\
0 \ar[r] & \ad{E^H_s} \ar[r]^{\iota_s} & \text{At}(E^H_s) \ar[r]^{\widetilde{d \varpi}_s} & T_{X_s} \ar[r] & 0 
}
\end{equation} 
where $r_s$ denotes the corresponding restriction map
for every $s \in S$,  and the bottom exact sequence is the Atiyah exact sequence for the principal $H$-bundle $E^H_s$ over $X_s$ (see \cite{A}).
The holomorphic splitting of the top exact sequence in 
\eqref{eq:cd1}
will induce a holomorphic splitting of the bottom exact 
sequence in \eqref{eq:cd1}.
 
 \end{proof}
 
 \subsection{Relative Atiyah bundle for group action}
 \label{group_action}
 Let $G$ be a connected complex Lie group acting holomorphically on $X$ and $S$ such that the holomorphic map $\pi : X \to S$ in \eqref{eq:0}
 is $G$-invariant. 
 
 Let 
 \begin{equation}
 \label{eq:10}
 \tau : G \times X \to X
 \end{equation}
 denote the action of $G$ on $X$ from left. Let $\mathfrak{g}$ denote the Lie algebra of $G$.
 
 Since $\pi : X \to S$ is $G$-invariant, 
 the differential of $\tau$ in \eqref{eq:10} induces an
 $\struct{X}$-linear morphism
 of vector bundles 
 \begin{equation}
 \label{eq:11}
 d'\tau : X \times \mathfrak{g} \longrightarrow T_{X/S}
 \end{equation}
 over $X$, where $X \times \mathfrak{g}$ is the trivial 
 vector bundle with fibre $\mathfrak{g}$ over $X$ and 
 $\struct{X}$ is the sheaf of holomorphic functions on $X$.
 Note that the image of $d'\tau$ need not be a subbundle of $T_{X/S}$.
 
 Define a holomorphic homomorphism of vector bundles 
 \begin{equation}
 \label{eq:12}
 \mu : \text{At}_S(E^H) \oplus (X \times \mathfrak{g})
 \longrightarrow T_{X/S}
 \end{equation}
over $X$ by 
 \begin{equation}
 \label{eq:13}
 \mu(u,v) = \widetilde{d \varpi} (u) - d'\tau(v),
 \end{equation}
 where $\widetilde{d \varpi}$ and $d' \tau$ are as given in equations \eqref{eq:9} and \eqref{eq:11} respectively.
 Since $\widetilde{d \varpi}$ is surjective, $\mu$ is 
 surjective. 
 
 Define an $\struct{X}$-submodule 
 \begin{equation}
 \label{eq:14}
 \text{At}_S^\tau(E^H) := \mu^{-1}(0) \subset \text{At}_S(E^H) \oplus (X \times \mathfrak{g}),
 \end{equation}
 which is in fact a subbundle, because $\widetilde{d \varpi}$ is surjective. 
 
 In view of definition of $\text{At}_S^\tau(E^H)$ in \eqref{eq:14}, we have two holomorphic homomorphisms
 \begin{equation}
 \label{eq:15}
 \iota_0 : \text{ad}(E^H) \longrightarrow \text{At}_S^\tau(E^H), \,\,\,\,\,\, u \mapsto (\iota(u), 0),
\end{equation}  
 where $\iota$ is defined in \eqref{eq:9}, and
 \begin{equation}
 \label{eq:16}
 q: \text{At}_S^\tau(E^H) \longrightarrow X \times \mathfrak{g}, \,\,\,\,\,\, (u, v) \mapsto v,
 \end{equation}
 where $u \in \text{At}_S(E^H)$ and $v \in X \times \mathfrak{g}$.
 
 Note that $q$ in \eqref{eq:16} is surjective because 
 $\widetilde{d \varpi}$ is surjective.
 
 Thus, we have a short exact sequence of holomorphic vector bundles over $X$
 \begin{equation}
 \label{eq:17}
 0 \,\longrightarrow\, \text{ad}(E^H)\, \stackrel{\iota_0}{\longrightarrow}\,
\text{At}_S^\tau(E^H) \, \stackrel{q }{\longrightarrow}\,  X \times \mathfrak{g} \,\longrightarrow\, 0
 \end{equation}
 
 A relative holomorphic $G$-connection on the principal 
 $H$-bundle $E^H$ is a holomorphic splitting of 
 \eqref{eq:17}, that is,
 there exists a holomorphic homomorphism of vector bundles
 \begin{equation*}
 \label{eq:18}
 h : X \times \mathfrak{g} \longrightarrow \text{At}_S^{\tau} (E^H)
 \end{equation*}
 such that 
 \begin{equation*}
 \label{eq:19}
 q \circ h = \id{X \times \mathfrak{g}}.
 \end{equation*}

\begin{remark}\label{conn-vs-G-conn}
It is easy to observe that the surjectivity of the map, $d'\tau : X \times \mathfrak{g} \longrightarrow T_{X/S} $, is a necessary condition for a  (relative) $G$-connection
to be a (relative) holomorphic connection. 
But, this is not a sufficient condition.
However, a (relative) $G$-connection corresponds to a (relative) holomorphic connection if  $d'\tau$ is an isomorphism; for instance, when the action $\tau: G \times X \to X$ is free and the dimensions of $G$ and $X_s$ are equal. When $d'\tau$ is an isomorphism, there is an isomorphism 
 $\phi: \text{At}_S (E^H) \to \text{At}_S^{\tau} (E^H) $ 
defined by $\phi(u) = (u, (d'\tau)^{-1} \widetilde{d \varpi} (u) )$. The inverse of $\phi$ has the simple formula, 
$\phi^{-1}(u,v) = u$. Moreover, given a splitting 
$\eta^{\tau}$ of \eqref{eq:17}, we have a splitting $\eta$ of the relative Atiyah sequence \eqref{eq:9} given by $$\eta(y) := \phi^{-1}(\eta^{\tau}(d'\tau)^{-1}(y))\, .$$ 
\end{remark}
 
 Let $V$ denote the trivial vector bundle $X \times \mathfrak{g}$ over $X$.
 Let $$\at[S]{E^H}_\tau \in \coh{1}{X}{V^* \otimes \text{ad}(E^H)}$$ be the extension class of the short exact 
 sequence \eqref{eq:17} and we call it the relative $G$-Atiyah 
 class of the principal $H$-bundle $E^H$ for the action 
 $\tau$.

 \begin{lemma}
\label{lem:2}
Let $G$ acts on $X$ and $S$ such that the morphism 
$\pi : X \to S$ is $G$-invariant.
Let $E^H \,\stackrel{\varpi}{\longrightarrow}\, X \xrightarrow{\pi} S$
be a holomorphic principal $H$-bundle. Then, the followings are equivalent.
\begin{enumerate}
\item $E^H$ admits a relative holomorphic $G$-connection.
\item The relative Atiyah exact sequence for $E^H$ in 
\eqref{eq:17} splits.
\item The relative $G$-Atiyah class $\at[S]{E^H}_{\tau}$ vanishes.
\end{enumerate}
\end{lemma}


 \begin{proposition} [family of holomorphic $G$-connections]
 \label{prop:2}
 Let $\pi : X \to S$ be $G$-invariant. Then, a relative 
 holomorphic $G$-connection on $E^H$ induces a family of 
 holomorphic $G$-connections on $\{E^H_s\}_{s \in S}$.
 \end{proposition}
 
 \begin{proof}
The proof is an easy consequence of the following commutative diagram
\begin{equation}
\label{eq:cd2}
\xymatrix{
0 \ar[r] & \ad{E^H} \ar[d]^{r_s} \ar[r]^{\iota_0} & \text{At}^\tau_S(E^H) 
\ar[d]^{r_s} \ar[r]^{q} & X \times \mathfrak{g} \ar[d]^{r_s} \ar[r] & 0 \\
0 \ar[r] & \ad{E^H_s} \ar[r]^{{\iota_0}_s} & \text{At}^\tau(E^H_s) \ar[r]^{q_s} & {X_s} \times \mathfrak{g} \ar[r] & 0 
}
\end{equation} 
where $r_s$ denotes the corresponding restriction map
for every $s \in S$,  and the bottom exact sequence is the Atiyah exact sequence of $E^H_s$ for the $G$ action on $X_s$  (see \cite{BP}).
The holomorphic splitting of the top exact sequence in 
\eqref{eq:cd2}
will induce a holomorphic splitting of the bottom exact 
sequence in \eqref{eq:cd2}. 
 \end{proof}
 
 \subsection{Sufficient condition for the existence of relative holomorphic connections}
 \label{suff_cond}
 We will give sufficient condition for the existence of
 relative holomorphic connections on $E^H$ and relative 
 holomorphic $G$-connections on $E^H$. In view of Proposition \ref{prop:1}, it is clear that a
 relative holomorphic connection on the principal $H$-bundle $E^H$ gives a family of 
 holomorphic connections. But the converse of Proposition \ref{prop:1} need not be true.

\begin{theorem}
\label{thm:1}
Let $E^H \,\stackrel{\varpi}{\longrightarrow}\, X$
be a holomorphic principal $H$-bundle. Suppose that for every
$s \in S$, there is a holomorphic connection on the principal $H$-bundle
$\varpi|_{E^H_s}\, :\, E^H_s \,\longrightarrow\, X_s$, and
$$\coh{1}{S}{\pi_*(\Omega^1_{X/S} \otimes \ad{E^H})} \,=\, 0\, .$$
Then, $E^H$ admits a relative holomorphic connection.
\end{theorem}
\begin{proof}
Consider the relative Atiyah exact sequence for the principal $H$-bundle in \eqref{eq:9}.  Tensoring it by
$\Omega^1_{X/S}$ produces the exact sequence
\begin{equation}\label{q}
0\,\longrightarrow\, \Omega^1_{X/S} \otimes \ad{E^H} \,\longrightarrow\, \Omega^1_{X/S} \otimes \text{At}_S(E^H)
\,\stackrel{q}{\longrightarrow}\,\Omega^1_{X/S}\otimes T_{X/S}
\,\longrightarrow\, 0\, ,
\end{equation}
where $q = \id{\Omega^1_{X/S}} \otimes \widetilde{d \varpi}$.

Note that $\struct{X}\cdot \id{T_{X/S}} \, \subset\, \text{End}(T_{X/S})\,=\,
\Omega^1_{X/S}\otimes T_{X/S}$. Define
$$
\Omega^1_{X/S} (\text{At}'_S(E^H))\, :=\, q^{-1}(\struct{X}\cdot \id{T_{X/S}})\, \subset\,
\Omega^1_{X/S}\otimes \text{At}_S(E^H)\, ,
$$
where $q$ is the projection in \eqref{q}. So we have the short exact sequence of sheaves
\begin{equation}\label{q2}
0\,\longrightarrow\, \Omega^1_{X/S}\otimes \ad{E^H}\,\longrightarrow\, \Omega^1_{X/S}(\text{At}'_S(E^H))
\,\stackrel{q}{\longrightarrow}\,\struct{X}\,\longrightarrow\, 0
\end{equation}
on $X$, where $\Omega^1_{X/S}(\text{At}'_S(E^H))$ is constructed above. Let
\begin{equation}\label{q3}
\Phi\, : \, {\rm H}^0(X,\, \struct{X})\, \longrightarrow\,
{\rm H}^1(X,\, \Omega^1_{X/S} \otimes \ad{E^H})
\end{equation}
be the connecting homomorphism in the long exact sequence of cohomologies associated to the exact sequence
in \eqref{q2}. The relative Atiyah class $\at[S]{E^H}$ (see \eqref{eq:9.3}) coincides
with $\Phi(1)\, \in\, {\rm H}^1(X,\, \Omega^1_{X/S} \otimes \ad{E^H})$. Therefore, from
Lemma \ref{lem:1} it follows that 
$E^H$ admits a relative holomorphic connection if and only if
\begin{equation}\label{q4}
\Phi(1)\, =\,0\, .
\end{equation}

To prove the vanishing statement in \eqref{q4}, first note that
${\rm H}^1(X,\, \Omega^1_{X/S}\otimes \ad{E^H})$ fits in the five terms exact sequence (Leray spectral sequence in low degrees)
\begin{equation*}
0 \to {\rm H}^1(S,\, \pi_*(\Omega^1_{X/S} \otimes \ad{E^H} ))\, \stackrel{\beta_1}{\longrightarrow}\,
{\rm H}^1(X,\, \Omega^1_{X/S} \otimes \ad{E^H})\,
\end{equation*}
\begin{equation}
\label{q5}
 \stackrel{q_1}{\longrightarrow}\,
{\rm H}^0(S,\, R^1\pi_*(\Omega^1_{X/S} \otimes \ad{E^H}))\, \stackrel{q_2}{\longrightarrow}\, \coh{2}{S}{ \pi_* \Omega^1_{X/S} \otimes \ad{E^H}} \stackrel{q_3}{\longrightarrow}\, \coh{2}{X}{\Omega^1_{X/S} \otimes \ad{E^H}},
\end{equation}
where $\pi$ is the projection of $X$ to $S$.
We use only first three terms in the above  exact 
sequence.

The given condition that for every
$s \in S$, there is a holomorphic connection on the holomorphic principal $H$-bundle
$\varpi|_{E^H_s}\, :\, E^H_s \,\longrightarrow\, X_s$, implies that
$$
q_1(\Phi(1))\,=\, 0\, ,
$$
where $q_1$ is the homomorphism in \eqref{q5}. Therefore, from the exact sequence in
\eqref{q5} we conclude that
$$
\Phi(1)\, \in\, \beta_1({\rm H}^1(S,\, \pi_*(\Omega^1_{X/S} \otimes \ad{E^H})))\, .
$$
Finally, the given condition that ${\rm H}^1(S,\, \pi_*(\Omega^1_{X/S} \otimes \ad{E^H}))\,=\, 0$
implies that $\Phi(1)\, =\, 0$. Since \eqref{q4} holds, the principal $H$-bundle 
$E^H$ admits a relative holomorphic connection.
\end{proof}

Next, consider the case of relative holomorphic $G$-connections on principal $H$-bundle $E^H \xrightarrow{\varpi} X \xrightarrow{\pi} S$. Under the assumption that the holomorphic map $\pi : X \to S$ is $G$-invariant, from Proposition \ref{prop:2},  a relative holomorphic $G$-connection
on $E^H$
gives a family of holomorphic $G$-connections on 
$\{E^H_s\}_{s \in S}$. Again, the converse of the Proposition \ref{prop:2} need not be true.

\begin{theorem}
\label{thm:2}
Let $E^H \,\stackrel{\varpi}{\longrightarrow}\, X \xrightarrow{\pi} S$
be a holomorphic principal $H$-bundle. Suppose that 
$\pi : X \to S$ is $G$-invariant. 
Let $V$ denote the trivial vector bundle $X \times \mathfrak{g}$ over $X$,  where $\mathfrak{g}$ is the Lie algebra of $G$.  Suppose that for every
$s \in S$, there is a holomorphic $G$-connection on the principal $H$-bundle
$\varpi|_{E^H_s}\, :\, E^H_s \,\longrightarrow\, X_s$, and
$$\coh{1}{S}{\pi_*(V^* \otimes\ad{E^H})} \,=\, 0\, ,$$
where $V^*$ denotes the dual of $V$.
Then, $E^H$ admits a relative holomorphic $G$-connection.
\end{theorem}
\begin{proof}
The proof is exactly similar to the proof of Theorem \ref{thm:1}. 
\end{proof}

\begin{example}\label{exmp:Stein}
Let $S$ be a Stein manifold. Let $\pi : X \to S$ be an analytic family of compact connected Riemann surfaces. Let $H$ be a connected reductive linear algebraic group over $\C$.
Let $E^H$ be a holomorphic principal $H$-bundle over $X$.
Suppose that for every $s \in S$, the principal $H$-bundle 
$E^H_s \to X_s$ admits a holomorphic connection (see \cite[Theorem 4.1]{AB} for the criterion of existence of holomorphic connection on $E^H_s$).  Since $S$ is Stein, and
$\pi_* (\Omega^1_{X/S} \otimes  \ad{E^H})$ is a coherent analytic sheaf, we have $\coh{1}{S}{\pi_*(\Omega^1_{X/S} \otimes \ad{E^H})} \,=\, 0\, .$ Hence by Theorem \ref{thm:1} , $E^H$ admits a relative holomorphic connection.

A similar reasoning can be given for the existence of relative holomorphic $G$-connection on $E^H$. In particular, let $G$ be the complex torus of dimension two, $X_s$ be a fixed smooth toric surface under $G$-action,  and $H =GL(r, \mathbb{C})$. Then it follows from \cite[Section 2.3, Example 4]{Kly} that $G$-equivariant principal $H$-bundles over $X_s$ are classified by families of filtrations of $\mathbb{C}^r$, each family being indexed by the torus invariant divisors of $X_s$. One can take $S$ to be a  suitable parameter space of such families isomorphic to the affine space with trivial $G$-action, and 
$X= X_s \times S$. This produces a nontrivial family $E^H$  of $G$-equivariant principal $H$-bundles over $X/S$. By Theorem \ref{thm:2}, and \cite[Theorem 3.1]{BKN} or \cite[Lemma 4.1]{BP}, $E^H$ admits a relative $G$-connection. 
\end{example}

\begin{example} 
\label{exmp:rel}
We give an example of existence of relative holomorphic $G$-connection on $E^H$, where the complex Lie group $H$ is abelian.
Let $F$ be a rank $n+1$ holomorphic vector bundle over  $\p^k$.
Then, in the notations of Theorem \ref{thm:2}, 
$$\pi : X = \p (F) \to \p^k = S$$ is a holomorphic flat morphism, which is a $\p^n$-bundle. Assume that the complex Lie group $H$ is abelian. Then, the adjoint vector bundle $\ad{E^H}$ is a trivial vector bundle over 
$X = \p (F)$.
For every $s \in S = \p^k$, $E^H_s$ admits a holomorphic $G$-connection (follows from the criterion \cite[Lemma 2.2]{BP}), because (see \cite[p.~5]{OSS}) $$\coh{1}{\p^n}{\ad{E^H_s}} = 0,$$
which follows from the fact that 
$\ad{E^H_s}$ is trivial.

Next, since $\pi$ is flat, the sheaf $\pi_*(V^* \otimes \ad{E^H})$ is a trivial vector bundle over $\p^k$,
and hence we get 
$$\coh{1}{\p^k}{\pi_*(V^* \otimes \ad{E^H})} = 0.$$
Therefore, from Theorem \ref{thm:2}, $E^H$ admits a relative holomorphic $G$-connection.

\end{example}

 \subsection{Curvature of relative holomorphic $G$-connection} 
 \label{curvature}
 Let $E^H \,\stackrel{\varpi}{\longrightarrow}\, X \xrightarrow{\pi} S$
be a holomorphic principal $H$-bundle and as in section 
\ref{group_action}, $G$ acts on $X$ and $S$ such that 
$\pi$ is $G$-invariant.
 Note that the sheaf of holomorphic sections of the vector bundle $\text{At}_S(E^H)$ has the Lie algebra structure.  Therefore, we get a Lie algebra structure 
 on the sheaf of holomorphic sections of the vector bundle $\text{At}^\tau_{S}(E^H)$, because 
 $$\text{At}_S^\tau(E^H) := \mu^{-1}(0) \subset \text{At}_S(E^H) \oplus (X \times \mathfrak{g})$$
 and $\mathfrak{g}$ is the Lie algebra of $G$.
Also, the morphisms $\iota_0$ and $q$ in \eqref{eq:17}
are compatible with the Lie bracket operations on the
sections of $\ad{E^H}$ and $\text{At}^{\tau}_S(E^H)$,
respectively.

 Let $\nabla : X \times \mathfrak{g} \longrightarrow \text{At}^\tau_S(E^H)$
 be a splitting of the short exact sequence in \eqref{eq:17}, that is, $\nabla$ is a relative holomorphic $G$-connection on $E^H$.
 Let $U \subset X$ be an open subset and  let $\alpha$ and 
 $\beta$ be  any two sections of $X \times \mathfrak{g}$
 over $U$. Consider 
 $$\cat{R}(\nabla) (\alpha, \beta) := [\nabla(\alpha), \nabla(\beta)] - \nabla ([\alpha, \beta]) \in \Gamma (U, \text{At}^\tau_S(E^H)) .$$
 Note that  $q (\cat{R}(\nabla)(\alpha, \beta)) = 0$,
 because $q$ in \eqref{eq:17} is compatible with the 
 Lie algebra structures.
 Hence $\cat{R}(\nabla) (\alpha, \beta)$ lies in the image of $\ad{E^H}$ over $U$.
 We also have following equalities
 \begin{enumerate}
 \item $\cat{R}(\nabla)(f \alpha, \beta) = f \cat{R}(\nabla)(\alpha, \beta),$ where $f$ is a holomorphic function on $U$.
 \item $\cat{R}(\nabla)(\alpha, \beta) = - \cat{R}(\nabla)(\beta, \alpha).$
\end{enumerate}  
Altogether, we get that 
\begin{equation}
\label{eq:20}
\cat{R}(\nabla) \in \coh{0}{X}{\ad{E^H} \otimes \bigwedge^2V^*} = \coh{0}{X}{\ad{E^H}} \otimes \bigwedge^2 \mathfrak{g}^*,
\end{equation}
where $V = X \times \mathfrak{g}$.
 
 The section $\cat{R}(\nabla)$ is called the {\it relative curvature} of the relative holomorphic $G$-connection $\nabla$ on 
 $E^H$.  
 
 Now, we describe the induced relative  connection
 and curvature for the holomorphic homomorphism of
 complex Lie groups.
 
 Let $\phi : G_1 \longrightarrow G $ be a holomorphic 
 homomorphism of complex Lie groups. Then, $G_1$ acts on 
 $X$ as follows
 \begin{equation}
 \label{eq:21}
 \tau_1 : G_1 \times X \longrightarrow  X, \,\,\,\, (g, x) \mapsto \tau (\phi(g), x)
\end{equation}   
where $\tau$ is the holomorphic action of $G$ on $X$ in \eqref{eq:10}.
Let $\mathfrak{g}_1$ denote the Lie algebra of $G_1$.
The differential of the morphism $\phi$ gives a homomorphism 
\begin{equation}
\label{eq:22}
d \phi : \mathfrak{g}_1 \longrightarrow \mathfrak{g}
\end{equation}
of Lie algebras. 

Note that for $G$ action on $X$ and $S$ such that 
$\pi: X \to S$  is $G$-invariant, the action $\tau_1$ in \eqref{eq:21} induces an action of $G_1$ on $S$ such that 
$\pi$ is $G_1$-invariant.  
Hence, we have a relative  Atiyah bundle $\text{At}^{\tau_1}_S(E^H)$ over $X$ as in \eqref{eq:14}. Since the action $G_1$ on $X$ and $S$ are given in terms of action of $G$ using the map 
$\phi$,  from the construction of $\text{At}^{\tau}_S(E^H)$ we have
\begin{equation}
\label{eq:23}
\text{At}^{\tau_1}_S(E^H) \, =\, \{ (u, v) \in \text{At}^\tau_S(E^H) \oplus (X \times \mathfrak{g}_1) \,\, \vert \,\, q(u) = (\id{X} \times d \phi)(v) \},
\end{equation}
where $q$ is given in \eqref{eq:17}.

An easy observation is stated as follows.

\begin{proposition}
\label{prop:3}
A relative holomorphic $G$-connection $\nabla$ on $E^H$
induces a relative holomorphic $G_1$-connection $\nabla_1$ on $E^H$. The relative curvature $\cat{R}(\nabla_1)$ coincides with the image of $\cat{R}(\nabla)$ under the homomorphism 
\begin{equation}
\label{eq:24}
\coh{0}{X}{\ad{E^H}} \otimes \bigwedge^2 \mathfrak{g}^*
\longrightarrow \coh{0}{X}{\ad{E^H}} \otimes \bigwedge^2 \mathfrak{g}_1^*
\end{equation}
induced by the dual homomorphism 
$(d \phi)^* : \mathfrak{g}^* \longrightarrow \mathfrak{g}^*_1$.
\end{proposition}

\section{Relative connection and lifting of an action}
\label{lift}

As in the previous section, let $G$ be a complex Lie group,
 $\pi : X \longrightarrow S$  equipped with the $G$-equivariant action, and $\varpi : E^H \to X $ be the family of principal $H$-bundles parametrised  by $S$.
 
 Henceforth, we will also assume that
 $X$ is compact.

Let  $\text{Aut}(E^H)$ be the group of all automorphisms
of $E^H$ over the identity map of $X/S$.
Because of the commutativity of the desired diagram,
given any $\theta \in \text{Aut}(E^H)$, we have an automorphism 
$\theta_s : E^H_s \to E^H_s$ over the identity map of $X_s$ for every $s \in S$.

Let 
\begin{equation}
\label{eq:24.1}
\nu : G \times S \to S
\end{equation}
denote the action of $G$ on $S$ such that for any $g \in G$
we have 
$$\pi \circ \tau_g = \nu_g \circ \pi,$$
where $\nu_g : S \longrightarrow S$ is an automorphism.
Then for every $s \in S$, we have an isomorphism 
\begin{equation}
\label{eq:27}
\tau_g : X_s \longrightarrow X_{\nu_g(s)}.
\end{equation}

Let $\text{Aut}_S(E^H)$ be the set of relative automorphisms, that is, $\text{Aut}_S(E_H)$ consists of those holomorphism automorphisms $\theta : E^H \longrightarrow E^H$ such that for every $s \in S$, 
$\theta_s : E^H_s \to E^H_{\nu_g(s)}$ is an isomorphism
over $\tau_g : X_s \to X_{\nu_g(s)}$. 

%
%

For any $g \in G$, we have an automorphism
\begin{equation}
\label{eq:26}
\tau_g : X \to X.
\end{equation}

Let 
\begin{equation}
\label{eq:28}
G_S \subset G
\end{equation}
be the subset consisting of all $g \in G$ such that for 
every $s \in S$ the pulled back principal $H$-bundle 
$\tau_g^* E^H_{\nu_g(s)}$ is isomorphic to $E^H_s$.

Note that if $\pi : X \to S$ is $G$-invariant, then 
for every $s \in S$, we get a subset $G_s$ of $G$,
consisting of all $g \in G$ such that $\tau^*_g E^H_s 
\cong  E^H_s$.

Let $\cat{G}_S$ denote the space of all pairs of the form 
$(\theta, g)$ where $g \in G_S$, and 
$\theta : E^H \longrightarrow E^H$ is a holomorphic automorphism such that for every $s \in S$, 
$\theta_s : E^H_s \to E^H_{\nu_g(s)}$ is an isomorphism
over $\tau_g : X_s \to X_{\nu_g(s)}$. 

Again, if $\pi : X \to S$ is $G$-invariant, then for 
every $s \in S$, we get a space $\cat{G}_s$ consisting of all pairs of the form $(\theta, g)$, where $g \in G_s$ and $\theta : E^H_s \to E^H_s$ is a holomorphic automorphism over the automorphism $\tau_g : X_s \to X_s$.

Note that $\cat{G}_S$ is equipped with the group operation defined as follows
\begin{equation}
\label{eq:29}
(\theta', g') \cdot (\theta, g) = (\theta' \circ \theta, g' g)
\end{equation}
while the inverse is the map 
$(\theta, g) \mapsto (\theta^{-1}, g^{-1})$.
Thus, $\cat{G}_S$ fits into the following short exact sequence of groups
\begin{equation}
\label{eq:30}
0 \to \text{Aut}(E^H) \,\xrightarrow{\alpha}\, \cat{G}_S \xrightarrow{\beta} G_S \to 0,
\end{equation}
where $\beta (\theta, g) = g$ and $\alpha(\theta) = (\theta, e)$, where $e$ is the identity of $G_S$.

There is a complex Lie group structure on $\cat{G}_S$
which is uniquely determined by the condition that 
\eqref{eq:30} is a sequence of complex Lie groups.

We already seen that the sheaf of sections of $\text{At}_S^{\tau}(E^H)$ admits a Lie algebra structure, and hence induces a Lie algebra structure on $\coh{0}{X}{\text{At}_S^\tau(E^H)}$.

\begin{proposition}
\label{prop:4} Suppose that $\pi : X \to S$ is $G$-invariant. Then,
the Lie algebra of $\cat{G}_S$ is canonically identified  with the above Lie algebra $\coh{0}{X}{\text{At}_S^\tau(E^H)}$.
\end{proposition}

\begin{proof}
Let $\mathfrak{g}_S$ denote the Lie algebra of $\cat{G}_S$. We shall produce a natural homomorphism 
from $\mathfrak{g}_S$ to $\coh{0}{X}{\text{At}_S^\tau(E^H)}$. 
 Note that the group $\cat{G}_S$ acts on $E^H$ naturally which commutes with the action $H$ on $E^H$. In fact, this action gives a map 
 from $\cat{G}_S$ to $\text{Aut}_S(E^H)$.
Consequently, we get a homomorphism of complex Lie algebras
\begin{equation}
\label{eq:30.5}
\eta : \mathfrak{g}_S \longrightarrow \coh{0}{X}{\text{At}_S(E^H)}.
\end{equation}
Next define
\begin{equation}
\label{eq:30.6}
\eta_1 : \mathfrak{g}_S \to \coh{0}{X}{\text{At}_S^\tau(E^H)}, \,\,\, v \mapsto (\eta(v), d\beta(v)) \in 
\coh{0}{X}{\text{At}_S(E^H)} \oplus \mathfrak{g},
\end{equation}
where $d \beta : \mathfrak{g}_S \to \text{Lie}(G_S) \hookrightarrow \mathfrak{g}$ is the homomorphism of Lie 
algebras associated to $\beta$ in \eqref{eq:30}.
It is easy to verify that 
$(\eta(v), d\beta(v)) \in \coh{0}{X}{\text{At}_S^\tau(E^H)} \subset
\coh{0}{X}{\text{At}_S(E^H)} \oplus \mathfrak{g}$ (see \eqref{eq:14} for the definition of $\text{At}_S^\tau(E^H)$). Clearly, $\eta_1$ in \eqref{eq:30.6}
is an injective homomorphism of complex Lie algebras.
Since $X$ is compact, using the similar statements 
as in  \cite[Proposition 3.1]{BP}, we can show that $\eta_1$ is surjective.

\end{proof}

Observe that there is a natural action of $\cat{G}_S$ 
on $X$ defined as follows
\begin{equation}
\label{eq:31}
\chi : \cat{G}_S \times X \to X \,\,\,\,\, ((\theta, g), x) \mapsto \tau(g,x), x \in X,
\end{equation}
where $\tau$ is the action in \eqref{eq:10}.
Then, we have a vector bundle $\text{At}^\chi_S(E^H)$
over $X$ as constructed in \eqref{eq:14}.

\begin{proposition}
\label{prop:5}
There is a natural isomorphism of vector bundles 
\begin{equation}
\label{eq:32}
\text{At}^{\chi}_S(E^H) \longrightarrow \ad{E^H} \oplus (X \times \coh{0}{X}{\text{At}^{\tau}_S(E^H)})
\end{equation}
where $X \times \coh{0}{X}{\text{At}^{\tau}_S(E^H)}$ is a trivial vector bundle on $X$ with fibre $\coh{0}{X}{\text{At}^{\tau}_S(E^H)}$.
\end{proposition}
\begin{proof}
Note that we have a natural projection 
$p_1 : \text{At}^\chi_S(E^H) \to X \times \mathfrak{g}_S$ from the short exact sequence as in \eqref{eq:17} for the group $\cat{G}_S$.
From previous Proposition \ref{prop:4}, $\mathfrak{g}_S$
is identified with 
$\coh{0}{X}{\text{At}^{\tau}_S(E^H)}$, therefore
we have 
$$p_1 : \text{At}^\chi_S(E^H) \longrightarrow X \times 
\coh{0}{X}{\text{At}^{\tau}_S(E^H)}.$$
Further, the action of $\cat{G}_S$ on $X$ factors 
through the action of $G$ on $X$, therefore from \eqref{eq:23} we 
have another description of $\text{At}^\chi_S(E^H)$ as 
subbundle of $\text{At}^{\tau}_S(E^H) \oplus (X \times \mathfrak{g}_S)$. Thus, we have a natural projection 
$$p' : \text{At}^{\tau}_S(E^H) \oplus (X \times \mathfrak{g}_S) = \text{At}^{\tau}_S(E^H) \oplus  (X \times 
\coh{0}{X}{\text{At}^{\tau}_S(E^H)} )\longrightarrow \text{At}^{\tau}_S(E^H)$$
which sends $(a, (x, \eta))\mapsto a - \eta (x) $, where 
$x \in X, a \in \text{At}^{\tau}_S(E^H)_x $ and $\eta \in \coh{0}{X}{\text{At}^{\tau}_S(E^H)}.$ 
From \eqref{eq:23}, it follows that 
$$q \circ (p' \vert_{\text{At}^\chi_S(E^H)}) = 0,$$
where $q$ is the projection in \eqref{eq:17}.
Therefore, the restriction $p' \vert_{\text{At}^\chi_S(E^H)}$ produces a homomorphism of vector bundles
\begin{equation}
\label{eq:33}
p_2 : \text{At}^\chi_S(E^H) \longrightarrow \ker{(q)} = \ad{E^H}.
\end{equation}
From $p_1$ and $p_2$, we get a homomorphism 
$$p_1 \oplus p_2 : \text{At}^\chi_S(E^H) \longrightarrow \ad{E^H} \oplus (X \times \coh{0}{X}{\text{At}^{\tau}_S(E^H)}),$$
which is an isomorphism.

\end{proof}

\begin{proposition}
\label{prop:6}
The holomorphic principal $H$-bundle $E^H$ admits a tautological
relative holomorphic $\cat{G}_S$-connection. The relative curvature of this  relative holomorphic $\cat{G}_S$-connection on $E^H$ vanishes identically.
\end{proposition}

\begin{proof}
Using the morphism $p_2$  in \eqref{eq:33},  Proposition easily follows. 
\end{proof}

\section{Relative equivariant bundles and relative connections}
\label{equi}
 In this section, we introduce the notion of  relative equivariance structure (see \cite{BKN} for the absolute case).
 Let $\pi : X \to S$ is $G$-equivariant map, where $G$ acts on $X$ and $S$ via  the actions $\tau$ in \eqref{eq:10}  and $\nu$ in \eqref{eq:24.1} respectively. 
 
 Let $E^H \,\stackrel{\varpi}{\longrightarrow}\, X \xrightarrow{\pi} S$ be a holomorphic principal $H$-bundle over $X/S$.
 
 A relative equivariance structure on the principal 
 $H$-bundle $E^H $ is a holomorphic action of $G$
 on the total space  $E^H$
 \begin{equation}
 \label{eq:34}
 \sigma^E : G \times E^H \longrightarrow E^H
 \end{equation}
 such that  the following diagrams commute.
\begin{enumerate}
\item 
\begin{equation*}
\xymatrix{
G \times E^H \ar[d]_{\id{G} \times \varpi} \ar[r]^{\sigma^E} & E^H \ar[d]^{\varpi } \\
G \times X \ar[d]_{\id{G} \times \pi} \ar[r]^{\tau} & X  \ar[d]^{\pi} \\ 
G \times S \ar[r]^{\nu} & S \\}
\end{equation*}
where $\varpi$ and $\tau$ are maps defined in 
\eqref{eq:1} and \eqref{eq:7}

\item 
\begin{equation*}
\xymatrix{
G \times E^H \times H \ar[d]_{\id{G} \times \sigma} \ar[r]^{\sigma^E \times \id{H}} & E^H \times H \ar[d]^{\sigma} \\
G \times E^H \ar[d]_{\id{G} \times \varpi} \ar[r]^{\sigma^E} & E^H  \ar[d]^{\varpi} \\ 
G \times X \ar[d]_{\id{G} \times \pi} \ar[r]^{\tau} & X  \ar[d]^{\pi} \\
G \times S \ar[r]^{\nu} & S \\
}
\end{equation*}
where $\sigma$ is defined in 
\eqref{eq:7}.
\end{enumerate}

A relative equivariant principal $H$-bundle is a principal $H$-bundle over $X/S$ with a relative equivariant structure.

Further, if $\pi : X \to S$ is $G$-invariant, then for every $s \in S$, the principal $H$-bundle $E^H_s$ over 
$X_s$ has a equivariant structure, that is, we have a family $\{E^H_s\}_{s \in S}$ of equivariant principal 
$H$-bundle parametrised by $S$.

\begin{proposition}
\label{prop:7} Suppose  $\pi : X \to S$ is $G$-invariant. Let $(E^H, \sigma^E)$ be a relative 
equivariant principal $H$-bundle over $X/S$. Then,
$E^H$ has a tautological relative holomorphic $G$-connection. The relative curvature of this relative  holomorphic $G$-connection vanishes identically.
\end{proposition}
\begin{proof}
Observe that for every $g \in G$, we have a holomorphic 
automorphism
$$\sigma^E_g : E^H \longrightarrow E^H$$
 such that for every $s \in S$, the map 
$$\sigma^E_{g,s} : E^H_s \longrightarrow E^H_{\nu_g(s)},  \,\,\,\,\,\, z \mapsto \sigma^E(g, z), \,\,\,\, z \in E^H_s$$ 
is an isomorphism over $\tau_g : X_s \longrightarrow X_{\nu_g(s)}$ in \eqref{eq:27}.
Thus, in this case $$G_S = G,$$ where $G_S$ is defined in
\eqref{eq:28}. 
In fact, we get a group homomorphism 
$$\beta^E : G_S = G \longrightarrow \cat{G}_S$$
defined by $g \mapsto (g, \sigma^E_g)$ such that 
$$\beta \circ \beta^E = \id{G},$$
where $\beta$ is the homomorphism in \eqref{eq:30}.
In view of Proposition \ref{prop:6}, we have a tautological relative holomorphic $\cat{G}_S$-connection
and we consider this. Next, using the above group 
homomorphism $\beta^E$ and Proposition \ref{prop:3}, we get a relative holomorphic $G$-connection on $E^H$.
Again from Proposition \ref{prop:6} and Proposition 
\ref{prop:3}, the relative curvature of this relative 
holomorphic $G$-connection vanishes identically. This completes the proof. 
\end{proof}

The following is the converse of Proposition \ref{prop:7}.
 
 \begin{proposition}
 \label{prop:8}
 Suppose that $\pi : X \to S$ is $G$-invariant.  Let 
 $h : X \times \mathfrak{g} \longrightarrow \text{At}^{\tau}_S(E^H)$ be a relative holomorphic $G$-connection
 on $E^H$ such that the relative curvature vanishes identically. Assume that $G$ is simply connected. Then, there exists a relative equivariant structure
 $$\sigma^E : G \times E^H \longrightarrow E^H$$
 such that the relative holomorphic $G$-connection associated to it as in Proposition \ref{prop:7} coincides with $h$.
 \end{proposition}
 \begin{proof}
 Since $\pi : X \to S$ is $G$-invariant, from Proposition \ref{prop:4}, we have
 $$\Lie{\cat{G}_S} = \mathfrak{g}_S =  \coh{0}{X}{\text{At}_S^\tau(E^H)}.$$
 Let 
 $$h_* : \mathfrak{g} =  \coh{0}{X}{ X \times \mathfrak{g}} \longrightarrow  \coh{0}{X}{\text{At}_S^\tau(E^H)} = \mathfrak{g}_S $$
 be the $\C$-linear map induced by $h$. Since, the relative curvature of the relative holomorphic $G$-connection $h$ vanishes identically, it follows that 
 $h_*$ is a homomorphism of Lie algebras.
 Further, since $G$ is simply connected, there is a unique holomorphic homomorphism of complex Lie groups 
 $$\epsilon : G \longrightarrow \cat{G}_S$$
 such that the differential 
 $$\text{d} \epsilon (1) : \mathfrak{g} \longrightarrow \mathfrak{g}_S$$ coincides with $h_*$, where $1$ denotes the identity element of $G$. 
 Recall that $\cat{G}_S$ acts naturally on $E^H$, and using the above holomorphic homomorphism $\epsilon$ of complex Lie groups, we produce a relative equivariant
 structure $\sigma^E$ on $E^H$. Now, observe that  the corresponding relative holomorphic $G$-connection given by Proposition 
 \ref{prop:7} coincides with $h$.   
 
 \end{proof}

\begin{theorem}
\label{thm:3}
Suppose that $\pi : X \to S$ is $G$-invariant.
Assume that $G$ is a semisimple and simply connected
affine algebraic group defined over $\C$. 
Let $E^H \,\stackrel{\varpi}{\longrightarrow}\, X \xrightarrow{\pi} S$
be a holomorphic principal $H$-bundle that admits a relative holomorphic $G$-connection $h$. Then, $E^H$
admits a relative equivariant structure 
$$\sigma^E : G \times E^H \longrightarrow E^H.$$
\end{theorem}

\begin{proof}
In view of Proposition \ref{prop:8}, it is enough to show that $E^H$ admits a relative holomorphic $G$-connection such that the relative curvature vanishes.
Now, consider a part of long exact sequence 
\begin{equation}
\label{eq:35}
\delta :  \mathfrak{g}_S = \coh{0}{X}{\text{At}_S^\tau(E^H)} \longrightarrow \coh{0}{X}{X \times \mathfrak{g  
}} = \mathfrak{g}
\end{equation}
associated to the short exact sequence \eqref{eq:17}.
Now, consider the homomorphism 
\begin{equation}
\label{eq:36}
h_* : \mathfrak{g} = \coh{0}{X}{X \times \mathfrak{g  
}} \longrightarrow \coh{0}{X}{\text{At}_S^\tau(E^H)} = 
\mathfrak{g}_S
\end{equation}
associated to the relative holomorphic $G$-connection $h$ on $E^H$. 
Since $E^H$ admits a relative holomorphic $G$-connection
$h$, we have 
$$\delta \circ h_* = \id{\mathfrak{g}},$$
and hence the Lie algebra homomorphism $\delta$  is 
surjective. As $G$ is semisimple, there exists a Lie 
subalgebra 
$$\mathfrak{h} \subset \coh{0}{X}{\text{At}_S^\tau(E^H)}$$ 
such that the restriction
$$\hat{\delta} := \delta \vert_\mathfrak{h} : \mathfrak{h} \longrightarrow \mathfrak{g} = \coh{0}{X}{X \times \mathfrak{g  
}}$$
is an isomorphism \cite[p. 91, Corollaire 3]{NB}.
We fix a subspace $\mathfrak{h}$ as above. Define 
$\widetilde{h_*}$  to be the following composition
$$\mathfrak{g} = \coh{0}{X}{X \times \mathfrak{g}} \xrightarrow{\hat{\delta}^{-1}} \mathfrak{h} 
\hookrightarrow  \coh{0}{X}{\text{At}_S^\tau(E^H)} =
\mathfrak{g}_S,$$
which is a Lie algebra homomorphism, and hence the relative curvature of the  relative holomorphic  
$G$-connection induced from $\widetilde{h_*}$ on $E^H$
vanishes identically. This completes the proof.

\end{proof}

\end{document}